\documentclass[11pt]{amsart}
\usepackage{graphicx}
\usepackage[mathscr]{eucal}
\usepackage{microtype}
\usepackage{amssymb}
\usepackage[parfill]{parskip}
\usepackage{mathrsfs}
\usepackage{xcolor}
\usepackage{color}
\usepackage[top=1in, bottom=1in, left=1in, right=1in]{geometry}
\usepackage{enumerate}
\usepackage{bm}
\usepackage[normalem]{ulem}
\usepackage[all]{xy}
\usepackage{tikz}

\usepackage{setspace} 

\newcommand{\rt}{\mathnormal{\mathsf{RT}}}
\newcommand{\zrt}{\mathnormal{\mathbb Z\text{-}\mathsf{RT}}}
\newcommand{\hind}{\mathnormal{\mathsf{HT}}}
\newcommand{\aht}{\mathnormal{\mathsf{AHT}}}
\newcommand{\apaht}{\mathnormal{\mathrm{ap}\mathsf{AHT}}}
\newcommand{\sepzrt}{\mathnormal{\mathrm{sep}\mathbb Z\text{-}\mathsf{RT}}}
\newcommand{\rca}{\mathnormal{\mathsf{RCA}}_0}
\newcommand{\aca}{\mathnormal{\mathsf{ACA}}_0}
\newcommand{\wkl}{\mathnormal{\mathsf{WKL}}_0}
\newcommand{\acaplus}{\mathnormal{\mathsf{ACA}}_0^+}

\newcommand{\weileq}{\mathop{\leq_{\mathsf{W}}}}
\newcommand{\weisleq}{\mathop{\leq_{\mathrm{s}\mathsf{W}}}}
\newcommand{\weieq}{\mathop{\equiv_{\mathsf{W}}}}
\newcommand{\weiseq}{\mathop{\equiv_{\mathrm{s}\mathsf{W}}}}
\newcommand{\isigma}{\mathnormal{\mathsf{I}\mathrm{\Sigma}_1^0}}

\DeclareMathOperator{\afs}{AFS}
\DeclareMathOperator{\fs}{FS}

\newtheorem{theorem}{Theorem}

\newtheorem{corollary}[theorem]{Corollary}
\newtheorem{question}[theorem]{Questions}

\theoremstyle{definition}
\newtheorem{definition}[theorem]{Definition}

\author[B. Aceves]{Bruno Fernando Aceves-Mart\'{\i}nez}
\address{
Escuela Superior de F\'{\i}sica y Matem\'aticas\\
Instituto Polit\'ecnico Nacional\\
Av. Instituto Polit\'ecnico Nacional s/n Edificio 9, 
Col. San Pedro Zacatenco, Alcald\'{\i}a Gustavo A. Madero, 07738, CDMX, Mexico. 
}
\curraddr{Departamento de Matem\'aticas, Centro de Investigaci\'on y Estudios Avanzados, CDMX, Mexico.}
\email{bfaceves@math.cinvestav.mx}

\author[D. Fern\'andez]{David J. Fern\'andez-Bret\'on}
\address{
Escuela Superior de F\'{\i}sica y Matem\'aticas\\
Instituto Polit\'ecnico Nacional\\
Av. Instituto Polit\'ecnico Nacional s/n Edificio 9, 
Col. San Pedro Zacatenco, Alcald\'{\i}a Gustavo A. Madero, 07738, CDMX, Mexico. 
}
\email{dfernandezb@ipn.mx}
\urladdr{https://dfernandezb.web.app}

\author[L. Romero]{L. F. Romero-Garc\'{\i}a}
\address{
Escuela Superior de F\'{\i}sica y Matem\'aticas\\
Instituto Polit\'ecnico Nacional\\
Av. Instituto Polit\'ecnico Nacional s/n Edificio 9, 
Col. San Pedro Zacatenco, Alcald\'{\i}a Gustavo A. Madero, 07738, CDMX, Mexico. 
}
\email{lromerog1801@alumno.ipn.mx}

\author[L. Villag\'omez]{Luis F. Villag\'omez-Canela}
\address{
Escuela Superior de F\'{\i}sica y Matem\'aticas\\
Instituto Polit\'ecnico Nacional\\
Av. Instituto Polit\'ecnico Nacional s/n Edificio 9, 
Col. San Pedro Zacatenco, Alcald\'{\i}a Gustavo A. Madero, 07738, CDMX, Mexico. 
}
\email{lvillagomezc1500@alumno.ipn.mx}

\title[Adjacent Hindman's and $\mathbb Z$-Ramsey's]{The adjacent Hindman's theorem and the $\mathbb Z$-Ramsey's theorem}



\begin{document}

\begin{abstract}
We consider the restriction of Ramsey's theorem that arises from considering only translation-invariant colourings of pairs, and show that this has the same strength (both from the viewpoint of Reverse Mathematics and from the viewpoint of Computability Theory) as the {\em Adjacent Hindman's Theorem}, proposed by L. Carlucci (Arch. Math. Log. {\bf 57} (2018), 381--359). We also investigate some higher dimensional versions of both of these statements.
\end{abstract}

\maketitle

\section{Introduction}

In this paper, we consider some Ramsey-theoretic combinatorial results from the perspective of both Computability Theory and Reverse Mathematics. Ramsey's theorem in dimension $n$, denoted $\rt^n$ ($n\geq 1$), is the statement that for every $c:[\mathbb N]^n\longrightarrow k$ (where $k$, an arbitrary finite number, is used to denote the set $\{0,\ldots,k-1\}$), there exists an infinite set $X\subseteq\mathbb N$ such that $[X]^n$ is $c$-monochromatic. On the other hand, Hindman's theorem, denoted $\hind$, is the statement that for every $c:\mathbb N\longrightarrow k$ there exists an infinite set $X\subseteq\mathbb N$ such that the set
\begin{equation*}
\fs(X)=\left\{\sum_{x\in F}x\bigg|F\subseteq X\text{ is finite and nonempty}\right\}
\end{equation*}
is $c$-monochromatic.

These two principles have been extensively studied, and constitute an important vein of contemporary research in the interface between combinatorics and (various branches of) logic. In order to mention some previous results, we proceed to briefly describe the systems involved. $\rca$ is a subsystem of Second-Order Arithmetic that, roughly speaking, has Peano Arithmetic as its first-order part, and includes closure under relative computability as its second-order part. $\wkl$ is $\rca$ plus closure under taking an infinite branch of each binary tree (i.e. the statement of K\"onig's Lemma, restricted to binary trees, holds), and $\aca$ is $\rca$ plus closure under taking Turing jumps. It is well-known that $\aca$ is strictly stronger than $\wkl$, which in turn is strictly stronger than $\rca$. The reader might consult~\cite{simpson-subsystems} for a thorough introduction to this area of research, or~\cite{hirschfeldt} for a treatment more readily oriented to the three systems that we just mentioned and their relationship to Ramsey theory in general. Nowadays, it is known that, if $n\geq 3$, then $\rt^n$ is equivalent to $\aca$ over $\rca$; somewhat surprisingly, $\rt^2$ is a principle at the same time strictly weaker than $\aca$ and incomparable with than $\wkl$, whereas $\rt^1$ is simply the pigeonhole principle (strictly stronger than plain $\rca$ but much weaker than $\wkl$). These results are classical, although the interested reader can consult, e.g.,~\cite{dzhafarov-mummert-textbook} (especially Corollaries 8.2.6 and 8.6.2) for a more contemporary source. On the other hand, it is also known that, again over $\rca$, $\acaplus$ implies $\hind$ which in turn implies $\aca$, where $\acaplus$ is essentially $\aca$ plus closure under the $\omega$-th Turing jump; these last results are due to Blass, Hirst and Simpson~\cite{blass-hirst-simpson}. The precise strength of $\hind$ in the hierarchy of reverse-mathematical principles is still unknown, and determining this strength is currently one of the important open problems in Reverse Mathematics. A good overview of the state of the art regarding this problem can be found in~\cite{carlucci-overview}.

With the finer distinctions provided by Computability Theory, especially the theory of Turing reducibility, one can be more precise about the results just mentioned. For example, for Ramsey's theorem, it is known~\cite{jockusch} that for every {\it instance} $c$ of the problem (that is, every colouring) there exists a solution (a set that is monochromatic for the colouring in question) $X$ such that $X\leq_T c^{(2n+2)}$, the $(2n+2)$-fold Turing jump of $c$; on the other hand, there is a computable instance $c$ of Ramsey's theorem such that every solution $X$ satisfies $\varnothing^{(n-2)}\leq_T X$. In the case of Hindman's theorem, it is known~\cite{blass-hirst-simpson} that there are computable instances $c$ to the theorem such that every monochromatic solution $X$ satisfies $\varnothing'\leq_T X$ whereas every instance $c$ of the theorem admits a solution $X$ such that $X\leq_T c^{(\omega+1)}$.

In this paper, we consider certain restrictions of Hindman's theorem and of Ramsey's theorem. On the front of Hindman's theorem, we shall work with L. Carlucci's {\it Adjacent Hindman's Theorem} proposed in~\cite{carlucci-adjacent}, which results from restricting the monochromaticity requirement not to a full set $\fs(X)$ but rather to the restricted set $\afs(X)$ of {\it adjacent} finite sums, meaning those sums of finitely many {\it consecutive} elements of $X$ when ordered increasingly, leaving no gaps. Carlucci showed that, over $\rca$, $\rt^2$ implies $\aht$, even with an extra condition that we will also consider, called {\it apartness}. On the other hand, we will consider a restriction of Ramsey's theorem that arises from restricting the allowable problems for which we require a solution ---concretely, stating the theorem only for those colourings that are invariant under the translation action of $\mathbb Z$, obtaining what we will call the {\it $\mathbb Z$-Ramsey's theorem}. Although this principle has not been previously studied within Reverse Mathematics, the idea of restricting instances of Ramsey's theorem only to translation-invariant colourings has already been used, e.g., by Petrenko and Protasov in their definition of $\mathbb Z$-Ramsey ultrafilters~\cite{petrenko-protasov}.

Many of the principles mentioned in the previous paragraph turn out to be equivalent to $\aca$ over $\rca$, so for the most part the viewpoint of Reverse Mathematics will not be very informative in that respect. In view of this, we focus more on the Computability-theoretic aspect of the principles, thinking instead about the {\it reducibility} relation, by means of which an instance of one problem can be turned into an instance of a different problem and a solution to the latter gets turned back into a solution to the original instance. More concretely, we will work with {\it Weihrauch reducibility}\footnote{Some authors (e.g.,~\cite[p. 22]{hirschfeldt}) have called this notion {\it uniform reducibility}, but the name ``Weihrauch reducibility'' seems to be the most commonly used nowadays. We are grateful to two anonymous referees for this observation.}. Recall~\cite[Definition 4.3.1]{dzhafarov-mummert-textbook} that a problem $P$ is {\bf Weihrauch reducible} to another problem $Q$, denoted $P\weileq Q$, if there are Turing functionals $\Phi,\Psi$ such that for every instance $X$ of $P$, $\Phi^X$ is an instance of $Q$, and for every solution $Y$ to $\Phi^X$, $\Psi^{X\oplus Y}$ is a solution to $X$. If we have the exact same situation, but with a Turing functional $\Psi$ such that $\Psi^Y$ is a solution to $X$ (i.e. we may compute the solution to $X$ just with information from $Y$ without needing to use any information about $X$), then we say that $P$ is {\bf strongly Weihrauch reducible}~\cite[Definition 4.4.1]{dzhafarov-mummert-textbook} to $Q$ and write $P\weisleq Q$. If each of the problems $P,Q$ is Weihrauch reducible (respectively, strongly Weihrauch reducible) to the other, then we say that $P$ is {\bf Weihrauch equivalent} (respectively, {\bf strongly Weihrauch equivalent}) to $Q$ and we denote this by $P\weieq Q$ (respectively, $P\weiseq Q$). So we will use this notion of reducibility to gauge the position of the Adjacent Hindman's Theorem and the $\mathbb Z$-Ramsey's Theorem within the computability-theoretic hierarchy; an important feature of all of our proofs is that they can be formalized within $\rca$, thus yielding corresponding (but coarser) Reverse Mathematics results as corollaries.

\section*{Acknowledgements}

We thank Lorenzo Carlucci for useful discussions and suggestions (including the realization that our proofs yielded not only Weihrauch reductions, but strong Weihrauch reductions), as well as two anonymous referees whose thorough and abundant comments helped greatly improve this paper. The second author was partially supported by IPN's internal grants SIP-20221862, SIP-20240886, and SIP-20253559, as well as by Secihti's grant CBF2023-2024-334. The first and fourth authors were supported by scholarships from Secihti.

\section{The $\mathbb Z$-invariant Ramsey's theorem}

\begin{definition}\hfill
\begin{enumerate}
\item Given an $n\geq 1$, we will say that a colouring of $n$-subsets $c:[\mathbb N]^n\longrightarrow k$ is {\bf $\mathbb Z$-invariant} if, for every $n$-subset $\{x_1,\ldots,x_n\}\in[\mathbb N]^n$ and for every $z\in\mathbb N$, it is the case that $c(\{x_1,\ldots,x_n\})=c(\{x_1+z,\ldots,x_n+z\})$.
\item The {\bf $\mathbb Z$-Ramsey's theorem for $n$-subsets}, which we will denote by $\zrt^n$, is the statement that every $\mathbb Z$-invariant colouring of $n$-subsets admits an infinite monochromatic set.
\end{enumerate}
\end{definition}

It is clear that, for each $n\geq 1$, $\zrt^n$ is a trivial corollary of $\rt^n$. Now, in the case of $\rt^n$, there is no difference (whether from the viewpoint of Weihrauch equivalence, or over $\rca$) whether one formulates it over $\mathbb N$ or over any infinite $X\subseteq\mathbb N$~\cite[Proposition 8.3.4]{dzhafarov-mummert-textbook}. This is not the case for $\zrt^n$: for example, if $X\subseteq\mathbb N$ is any infinite set with the property that (if we enumerate its elements increasingly by $x_n$) $x_{n+1}-x_n>x_n-x_0$ for all $n$, then (since no two distinct pairs of elements of $X$ are at the same distance from one another) the family of $\mathbb Z$-invariant colourings restricted to $X$ is the same as the family of all arbitrary colourings on $X$. Therefore, for such a set, the version of $\zrt^n$ stated for $X$ would be Weihrauch-equivalent to the full $\rt^n$. We therefore refrain from considering ``general'' versions of the $\mathbb Z$-Ramsey theorem.

The following theorem basically establishes that $\zrt^{n+1}$ implies $\rt^n$ over $\rca$, although, as we warned in the introduction, we provide more precise information by phrasing our theorem in the language of Weihrauch reducibility.

\begin{theorem}\label{thm2}
For each $n\geq 1$, $\rt^n\weisleq\zrt^{n+1}$.
\end{theorem}

\begin{proof}
The witnessing Turing functionals, which we denote by $\Phi$ and $\Psi$, are as follows: for a colouring $c:[\mathbb N]^n\longrightarrow 2$ we let $\Phi^c:[\mathbb N]^{n+1}\longrightarrow 2$ be given by
\begin{equation*}
\Phi^c(x_0,x_1,\ldots,x_n)=c(x_1-x_0,x_2-x_0,\ldots,x_n-x_0).
\end{equation*}
It is easily verified that $\Phi^c$ is $\mathbb Z$-invariant. In the other direction, given an infinite set $A\subseteq\mathbb N$, we let
\begin{equation*}
\Psi^A=\{x-x_0\big|x\in A\setminus\{x_0\}\},
\end{equation*}
where $x_0=\min(A)$. If $[A]^{n+1}$ is $\Phi^c$-monochromatic (say, in colour $i$) then $[\Psi^A]^n$ is $c$-monochromatic (in the same colour $i$), since for any elements $x_1-x_0,\ldots,x_n-x_0\in \Psi^A$ we have
\begin{equation*}
c(x_1-x_0,\ldots,x_n-x_0)=\Phi^c(x_0,x_1,\ldots,x_n)=i.
\end{equation*}
\end{proof}

The attentive reader should note that, in the previous proof, in order for $[\Psi^A]^n$ to be $c$-monochromatic, we do not fully use the fact that all $(n+1)$-subsets of $A$ are $\Psi^c$-monochromatic, but only those whose first element is $x_0$---for example, for $n=1$, we only need the $2$-subsets taking their first element from $\{x_0\}$ and their second element from $\{x_n\big|n\geq 1\}$, which looks like a very particular case of a Polarized version of Ramsey's Theorem. This suggests an intriguing question for possible further research, namely whether some restriction of the Increasing Polarized Ramsey's Theorem, possibly in a higher dimension, can already imply Ramsey's Theorem---and what would be the right restriction to make in order to obtain an equivalence.

The proof of Theorem~\ref{thm2} is simple enough that, by carefully going through it from the perspective of Reverse Mathematics, it is apparent that no extra induction beyond $\isigma$ is used. Hence, this proof also yields implications over $\rca$, which we will summarize in the following corollary. Since the statements $\rt^n$ are equivalent for $n\geq 3$ (and they are equivalent to $\aca$), the only interesting corollary, from the perspective of Reverse Mathematics, is the part where $n\leq 3$.

\begin{corollary}\label{cor:implicaciones-reverse}
\hfill
\begin{enumerate}
\item For each $n\geq 4$, $\rca\vdash\zrt^n\leftrightarrow\aca$,
\item $\aca\vdash\zrt^3$,
\item $\rca\vdash\zrt^3\rightarrow\rt^2\rightarrow\zrt^2\rightarrow\rt^1$.
\end{enumerate}
\end{corollary}

{\it A priori}, the arrows on part (3) of Corollary~\ref{cor:implicaciones-reverse} may or may not be reversible. Later on, in the last section (and once we have more tools under our belt), we will look into this specific question regarding the arrow from $\zrt^3$ to $\rt^2$.

\subsection{Relation with the Adjacent Hindman's Theorem}

Let us begin by recalling Carlucci's adjacent Hindman's theorem.

\begin{definition}
The {\bf adjacent Hindman's theorem}, denoted by $\aht$, is the statement that, for every finite colouring $c:\mathbb N\longrightarrow k$, there exists an infinite set $X$ such that the set
\begin{equation*}
\afs(X)=\{x_n+x_{n+1}+\cdots+x_{n+l}\big|n,l\in\mathbb N\}
\end{equation*}
is monochromatic, where the sequence $\langle x_n\big|n\in\mathbb N\rangle$ represents the (unique) increasing enumeration of the set $X$.
\end{definition}

Given that the set $\afs(X)$ is defined in such a way that the order within our set matters (unlike the usual $\fs$ sets utilized in Hindman's theorem), one could conceivably state a version of the $\aht$ for {\it sequences} rather than sets, i.e. affirming the existence of a sequence $\langle x_n\big|n\in\mathbb N\rangle$, not necessarily increasing (or even injective!) whose adjacent finite sums are monochromatic. However, it is readily seen that these two versions of the $\aht$ must be equivalent (whether from the viewpoint of Reverse Mathematics over $\rca$, or from the viewpoint of Computability Theory). Clearly the original version as stated by Carlucci implies the ``sequence'' version (by replacing the set $X$ with the sequence that increasingly enumerates its elements); conversely, given an arbitrary sequence $x_n$, one can recursively define $y_1=x_1$, $k_1=1$; and $y_{n+1}=x_{k_n+1}+x_{k_n+2}+\cdots+x_{k_{n+1}}$, with $k_{n+1}$ the least number making the $y_{n+1}$ thus defined to be strictly larger than $y_n$. This way we obtain a strictly increasing sequence $\langle y_n\big|n\in\mathbb N\rangle$ (in particular, we obtain a set $Y$ whose increasing enumeration is precisely the sequence of $y_n$) such that $\afs(Y)=\afs(y_n\big|n\in\mathbb N)\subseteq\afs(x_n\big|n\in\mathbb N)$.

The following theorem is the observation that originally launched the work in this paper.

\begin{theorem}\label{zrtversusaht}
$\zrt^2\weiseq\aht$.
\end{theorem}

\begin{proof}
For the first reducibility relation, we describe the corresponding Turing functionals $\Phi_1,\Psi_1$ as follows.
Given a $\mathbb Z$-invariant colouring $c:[\mathbb N]^2\longrightarrow k$, $\Phi_1^c:\mathbb N\longrightarrow k$ is defined by $\Phi_1^c(y)=c(0,y)$. Note that the assumption that $c$ is $\mathbb Z$-invariant implies $\Phi_1^c(y)=c(x,x+y)$ for every natural number $x$. On the other hand, for any infinite set $Y=\{y_n\big|n\in\mathbb N\}$ (where the indexing $y_n$ of the elements of $Y$ constitutes an increasing enumeration), we let $\Psi_1^Y=\{y_1+\cdots+y_n\big|n\in\mathbb N\}$. For every pair of elements of $\Psi_1^Y$, we have
\begin{eqnarray*}
c(y_1+\cdots+y_n,y_1+\cdots+y_m) & = & c(y_1+\cdots+y_n,(y_1+\cdots+y_n)+(y_{n+1}+\cdots+y_m)) \\
 & = & c(0,y_{n+1}+\cdots+y_m) \\
 & = & \Phi_1^c(y_{n+1}+\cdots+y_m).
\end{eqnarray*}
Since numbers of the form $y_{n+1}+\cdots+y_m$ are precisely the elements of $\afs(Y)$, the conclusion is that $\afs(Y)$ is $\Phi_1^c$-monochromatic if and only if $[\Psi_1^Y]^2$ is $c$-monochromatic.

For the converse reducibility relation, we denote the relevant functionals by $\Phi_2,\Psi_2$. Given a $d:\mathbb N\longrightarrow k$ we let $\Phi_2^d:[\mathbb N]^2\longrightarrow k$ be given by $\Phi_2^d(\{x,y\})=d(y-x)$, whenever $x<y$; it is readily seen that $\Phi_2^d$ is a $\mathbb Z$-invariant colouring. On the other hand, for an infinite set $X$, first define recursively $X'=\{x_n\big|n\in\mathbb N\}$ by letting $x_0,x_1$ be the two smallest elements of $X$, and then we let each $x_{n+1}$ be the least element of $X$ so that $x_{n+1}-x_n>x_n-x_{n-1}$. We then let $\Psi_2^X=\{x_{n+1}-x_n\big|n\in\mathbb N\}$. Note that an element of $\afs(\Psi_2^X)$ is of the form
\begin{equation*}
(x_{n+1}-x_n)+(x_{n+2}-x_{n+1})+\cdots+(x_{n+k+1}-x_{n+k})=x_{n+k+1}-x_n,
\end{equation*}
therefore if $x_{n+k+1}-x_n=y\in\afs(\Psi_2^X)$ then $d(y)=d(x_{n+k+1}-x_n)=\Phi_2^d(x_n,x_{n+k+1})$; so, if $X$ is $\Phi_2^d$-monochromatic then $\afs(\Psi_2^X)$ is $d$-monochromatic.
\end{proof}

Note that the previous proof, in a sense, shows that $\mathbb Z$-invariant colourings are precisely those colourings that depend only on the distance between the two elements of the pair being coloured. Once again, the proof is simple enough that one can directly see that no induction beyond $\isigma$ is used; therefore, the same proof (appropriately formalized) yields as a result the following corollary.

\begin{corollary}
$\rca\vdash\aht\leftrightarrow\zrt^2$.
\end{corollary}

\section{Adjacent Hindman's theorem in higher dimensions.}

We now proceed to study a higher-dimensional version of the Adjacent Hindman's Theorem, much in the same spirit that the Milliken--Taylor theorem constitutes a higher-dimensional version of the usual Hindman's theorem. Recall first that Hindman's theorem has an equivalent formulation in terms of finite unions: for every colouring of the set of finite subsets of $\mathbb N$, there exists an infinite, pairwise disjoint family of such finite subsets satisfying that all of the unions of finitely many of those sets receive the same colour. The Milliken--Taylor theorem~\cite{milliken-thm,taylor-thm}, for example for dimension $2$, ensures that, given any colouring of pairs of finite subsets of $\mathbb N$, there exists an infinite pairwise disjoint family such that all {\it ordered} pairs of finite unions have the same colour\footnote{A pair of finite sets $(a,b)$ is ordered if $\max(a)<\min(b)$. We switched from the finite-sums to the finite-unions formulation here in order to be able to use this notion; the Milliken--Taylor theorem can also be formulated in terms of finite-sums but then one must introduce the $\lambda$ and $\mu$ functions, which we will only do in the next subsection.}; and analogously for higher-dimensional versions. In the adjacent context, we will have to consider {\it adjacent} pairs of adjacent finite unions, and similarly for higher dimensions.

\begin{definition}\hfill
\begin{enumerate}
\item Let $\vec{x}=\langle x_n\big|n\in\mathbb N\rangle$ be a sequence of natural numbers, and let $d\in\mathbb N\setminus\{0\}$. We define the set of {\bf adjacent $d$-tuples of adjacent sums from $\vec{x}$} to be the set
\begin{equation*}
\afs^d(\vec{x})=\left\{\left(\sum_{k=k_0}^{k_1} x_k,\sum_{k=k_1+1}^{k_2} x_k,\ldots,\sum_{k=k_{d-1}+1}^{k_d} x_k\right)\bigg|k_0\leq k_1<k_2<\cdots<k_d\right\}
\end{equation*}
\item We define the {\bf $d$-Adjacent Hindman's Theorem}, denoted $\aht^d$, to be the statement that for every colouring $d:\mathbb N^d\longrightarrow k$ there exists an infinite set $Y\subseteq\mathbb N$ such that, if $\vec{y}=\langle y_n\big|n\in\mathbb N\rangle$ is the increasing enumeration of $Y$, then the set $\afs^d(\vec{y})$ is $d$-monochromatic.
\end{enumerate}
\end{definition}

Just as in the case of the $1$-dimensional adjacent Hindman's theorem, $\aht^d$ for $d>1$  could be phrased directly in terms of sequences, and we would still obtain an equivalent statement. The following is the generalization of Theorem~\ref{zrtversusaht} to this context.

\begin{theorem}\label{higher-zrtversusaht}
For each $d\in\mathbb N\setminus\{0\}$, we have $\zrt^{d+1}\weiseq\aht^d$.
\end{theorem}

\begin{proof}
The Turing functionals $\Phi_1,\Psi_1$ witnessing that $\zrt^{d+1}\weisleq\aht^d$ are: for a $\mathbb Z$-invariant colouring $c:[\mathbb N]^{d+1}\longrightarrow k$, we let $\Phi_1^c:\mathbb N^d\longrightarrow k$ be given by
\begin{equation*}
\Phi_1^c(y_1,\ldots,y_d)=c(0,y_1,y_1+y_2,\ldots,y_1+y_2+\cdots+y_d).
\end{equation*}
On the other hand, given an infinite set $Y=\{y_n\big|n\in\mathbb N\}$, enumerated increasingly, let $\Psi_1^Y=\{y_1+\cdots+y_n\big|n\in\mathbb N\}$. 
Now, if $x_0,\ldots,x_d\in\Psi_1^Y$ are distinct elements, with $x_0<\cdots<x_d$, then we have 
$x_i=y_1+\cdots+y_{k_i}$ for each $i$, with
$k_0<k_2<\cdots<k_d$. Then, by the $\mathbb Z$-invariance of $c$ we have
\begin{eqnarray*}
c(x_0,\ldots,x_d) & = & c(0,x_1-x_0,x_2-x_0,\ldots,x_d-x_0) \\
 & = & c(0,z_1,z_1+z_2,\ldots,z_1+z_2+\cdots+z_d) \\
 & = & \Phi_1^c(z_1,z_2,\ldots,z_d),
\end{eqnarray*}
where we have defined $z_i=x_i-x_{i-1}=y_{k_{i-1}+1}+\cdots+y_{k_i}$. So the values of $c$ on $(d+1)$-subsets from $\Psi_1^Y$ are exactly the values of $\Phi_1^c$ on $d$-tuples from $Y$ that have the form
\begin{equation*}
(y_{k_0+1}+\cdots+y_{k_1},y_{k_1+1}+\cdots+y_{k_2},\cdots,y_{k_{d-1}+1}+\cdots+y_{k_d}),
\end{equation*}
which are precisely the elements of $\afs^d(Y)$. Hence $\Psi_1^Y$ is $c$-monochromatic if and only if $Y$ is $\Phi_1^c$-monochromatic.

Now, to prove $\aht^d\weisleq\zrt^{d+1}$, the witnessing Turing functionals will be $\Phi_2,\Psi_2$ defined as follows: for $d:\mathbb N^d\longrightarrow k$ we let $\Phi_2^d:[\mathbb N]^{d+1}\longrightarrow k$ be given by $\Phi_2^d(x_0,\ldots,x_d)=d(x_1-x_0,x_2-x_1,\ldots,x_d-x_{d-1})$ (if $x_1<\ldots<x_{d+1}$); it is easy to see that $\Phi_2^d$ is a $\mathbb Z$-invariant colouring. For the definition of $\Psi_2$ we proceed as in the proof of theorem~\ref{zrtversusaht}: given an infinite set $X$, first define $X'=\{x_n\big|n\in\mathbb N\}$ such that $x_0,x_1$ are the two smallest elements of $X$, and $x_{n+1}$ is the least element of $X$ so that $x_{n+1}-x_n>x_n-x_{n-1}$ for each $n\geq 1$. We then let $\Psi_2^X=\{x_{n+1}-x_n\big|n\in\mathbb N\}$. Note that, if $y_j=x_{j+1}-x_j$, then an element $\vec{y}\in\afs^d(\Psi_2^X)$ is of the form
\begin{eqnarray*}
\vec{y} & = & (y_{i_0}+\cdots+y_{i_{1}-1},y_{i_1}+\cdots+y_{i_2-1},\ldots,y_{i_{d-1}}+\cdots+y_{i_d-1}) \\
& = & (x_{i_1}-x_{i_0},x_{i_2}-x_{i_1},\ldots,x_{i_d}-x_{i_{d-1}}), 
\end{eqnarray*}
so that $d(\vec{y})=\Phi_2^d(x_0,\ldots,x_d)$ and hence, if $[X]^{n+1}$ is $\Phi_2^d$-monochromatic this implies that $\afs(\Psi_2^X)$ is $d$-monochromatic.
\end{proof}

Once again, by going carefully through the previous proof, and noticing that (the proof is simple enough that) no induction beyond $\isigma$ is used, one proves the following corollary.

\begin{corollary}
For each $d$, $\rca\vdash\zrt^{d+1}\leftrightarrow\aht^d$.
\end{corollary}

\subsection{The apartness condition}

In Carlucci's paper~\cite{carlucci-adjacent} (where the original formulation of the $\aht$ can be found), a version of the $\aht$ is considered where one requires an {\it apartness} condition on the generators of the monochromatic set. In order to understand this requirement, recall that, after~\cite{blass-hirst-simpson}, for an $x\in\mathbb N$ we let $\lambda(x)$ be the smallest position, and we let $\mu(x)$ be the largest position, where a nonzero digit appears if $x$ is written in binary notation (when reading the number from right to left, i.e., from the least significant to the most significant digit). Equivalently, and even more intuitively, one can simply think of the natural number $n$ as a finite subset of $\mathbb N$ (by letting the binary expansion of $n$ represent the characteristic function of such finite subset) and then $\lambda(n)$ and $\mu(n)$ are simply the minimum and maximum elements of this finite set. There are other ways to describe these couple of functions (for example, $\lambda(x)$ is the unique $n$ such that $2^n\mid x$ but $2^{n+1}\nmid x$, and $\mu(x)=\lfloor\log_2(x)\rfloor$), but the description in terms of the binary representation is much easier to visualize. Using this tool, we are able to state the following definition (which is already implicit in~\cite{blass-hirst-simpson}, and more explicit in~\cite[Definition 3]{carlucci-adjacent}, but for sets; we simply render below the same definition but for sequences).

\begin{definition}
A sequence $\langle x_n\big|n\in\mathbb N\rangle$ is said to satisfy {\it the apartness condition} if, for every $n\in\mathbb N$, we have $\mu(x_n)<\lambda(x_{n+1})$.
\end{definition}

The reader familiar with the definition of a {\it block sequence} on $[\mathbb N]^{<\aleph_0}$ will recognize that a sequence of natural numbers satisfies the apartness condition precisely when the finite sets corresponding to the elements of the sequence (as described in the previous paragraph) form a block sequence. Carlucci's {\it adjacent Hindman's theorem with apartness} is the principle stating the same as $\aht$ but requiring that the sequence whose set of adjacent finite sums is monochromatic satisfy the apartness condition; we will henceforth denote this principle with the symbol $\apaht$. We similarly define the higher-dimensional versions $\apaht^d$ for each $d$.

Recall the proof of Theorem~\ref{zrtversusaht}. Looking at the definitions of the relevant Turing functionals $\Psi_1,\Psi_2$, the reader will notice that, if the infinite set $Y=\{y_n\big|n\in\mathbb N\}$ satisfies the apartness condition, then the set $\Psi_1^Y$ will have an increasing enumeration $\{x_n\big|n<\mathbb N\}$ satisfying $\mu(x_n-x_{n-1})<\lambda(x_{n+1}-x_n)$. Conversely, given an infinite set $X=\{x_n\big|n<\mathbb N\}$ (increasing enumeration) satisfying $\mu(x_n-x_{n-1})<\lambda(x_{n+1}-x_n)$, then the set $\Psi_2^X$, as defined in the proof of Theorem~\ref{zrtversusaht}, which is simply $\{x_{n+1}-x_n\big|n\in\mathbb N\}$ (note that, in this case, the auxiliary subset $X'\subseteq X$ is simply $X$ itself), will satisfy the apartness condition. This motivates the following definition of a condition that is to versions of Ramsey's theorem what the apartness condition is to versions of Hindman's theorem.

\begin{definition}\hfill
\begin{enumerate}
\item A set $X$ is said to satisfy the {\bf separation condition} if, letting $\langle x_n\big|n\in\mathbb N\rangle$ be the increasing enumeration of $X$, we have $\mu(x_n-x_{n-1})<\lambda(x_{n+1}-x_n)$ for every $n\in\mathbb N$.  
\item Given an $n\in\mathbb N$, the {\bf $\mathbb Z$-Ramsey's theorem for $n$-subsets with separation} is the statement that every $\mathbb Z$-invariant colouring of $n$-subsets admits an infinite monochromatic set that, when enumerated increasingly, yields a sequence satisfying the separation condition. We denote this principle with the symbol $\sepzrt^n$.
\end{enumerate}
\end{definition}

Notice that the separation condition is not hereditary to subsets, i.e., if $Y\subseteq X$ then $Y$ does not necessarily satisfy the separation condition even if $X$ does; therefore, $\sepzrt$ is {\it a priori} a stronger statement than $\zrt$ (it cannot be obtained simply by applying $\zrt$ to a set that satisfies the separation condition; furthermore, it is not even clear that it can be obtained by first applying $\zrt$ and then thinning out the resulting monochromatic set). More concretely, $\sepzrt^n$ implies $\zrt^n$ in $\rca$ or, in terms of Weihrauch reducibility, we have $\zrt^n\leq\sepzrt^n$. Looking carefully (whether literally as Computability statements, or as proofs in $\rca$) at the proofs from Theorem~\ref{higher-zrtversusaht} (as described two paragraphs above but now in this more general context), yields the following.

\begin{theorem}\label{adj-zrtversusaht}\hfill
\begin{enumerate}
\item $\sepzrt^{d+1}\weiseq\apaht^d$,
\item for each $d$, the statements $\sepzrt^{d+1}$ and $\apaht^d$ are equivalent over $\rca$.
\end{enumerate}
\end{theorem}

For a relationship between the versions of $\zrt$ with or without the separation condition (or, equivalently, between the versions of $\aht$ with or without apartness), Carlucci~\cite{carlucci-adjacent} proves that $\rt^2$ implies $\apaht$. Using the same reasoning, we prove the higher-dimensional general version of this result.

\begin{theorem}
For each $n\in\mathbb N$, we have $\apaht^n\weisleq\rt^{n+1}$.
\end{theorem}

\begin{proof}
Consider the Turing functionals $\Phi$ and $\Psi$ given as follows: for a colouring $c:[\mathbb N]^n\longrightarrow k$, let $\Phi^c:[\mathbb N]^{n+1}\longrightarrow k$ be given by
\begin{equation*}
\Phi^c(x_0,\ldots,x_n)=c(2^{x_1}+2^{x_1+1}+\cdots+2^{x_2-1},2^{x_2}+2^{x_2+1}+\cdots+2^{x_3-1},\ldots,2^{x_{n-1}}+2^{x_{n-1}+1}+\cdots+2^{x_n}).
\end{equation*}
On the other hand, given a $Y=\{y_n\big|n\in\mathbb N\}$ enumerated increasingly, let 
\begin{equation*}
\Psi^Y=\{2^{y_n}+2^{y_n+1}+\cdots+2^{y_{n+1}}\big|n\in\mathbb N\}.
\end{equation*}
Then $[Y]^{n+1}$ is $\Phi^c$-monochromatic if and only if $\afs^n(\Psi^Y)$ is $c$-monochromatic.
\end{proof}

This way, we obtain the following diagram, where the arrows can mean either implication under $\rca$, or Weihrauch reducibility (and we make no further claims as to whether the arrows are reversible or not; some of these collapse under $\rca$ but not necessarily under Weihrauch reduction)

\begin{equation*}
\centerline{\xymatrix{
\cdots \ar@{->}[r] & \sepzrt^{n+1} \ar@{->}[r] \ar@{<->}[d] & \zrt^{n+1} \ar@{->}[r] \ar@{<->}[d] & \rt^n \ar@{->}[r] & \sepzrt^n \ar@{->}[r] \ar@{<->}[d] & \zrt^n \ar@{->}[r] \ar@{<->}[d] & \rt^{n-1} \ar@{->}[r] & \cdots \\
\cdots & \apaht^n \ar@{->}[r] & \aht^n &  & \apaht^{n-1} \ar@{->}[r] & \aht^{n-1} &  & \cdots 
}}
\end{equation*}

Towards the right-extreme of this diagram, we find the following configuration:

\begin{equation*}
\centerline{\xymatrix{
\cdots \ar@{->}[r] & \rt^3 \ar@{->}[r] & \sepzrt^3 \ar@{->}[r] \ar@{<->}[d] & \zrt^3 \ar@{->}[r] \ar@{<->}[d] & \rt^2 \ar@{->}[r] & \sepzrt^2 \ar@{->}[r] \ar@{<->}[d] & \zrt^2 \ar@{->}[r] \ar@{<->}[d] & \rt^1\ar@{<->}[d] \\
\cdots &  & \apaht^2 \ar@{->}[r] & \aht^2 &  & \apaht \ar@{->}[r]  & \aht  & \text{pigeonhole} \\
}}
\end{equation*}

(Other principles could be added to the diagram, e.g., the fact that $\apaht$ implies $\mathsf{D}^2$~\cite{carlucci-adjacent}, but we only added the principles explicitly studied in this paper to the diagram.) This is very informative from the perspective of Weihrauch reducibility. From the perspective of Reverse Mathematics, on the other hand, a huge portion of the diagram collapses, since it is known that $\rt^3$ already implies $\aca$. Therefore, the first diagram (up to $n=4$) consists of equivalent statements, and the only question remaining should be regarding the second diagram. The next (and last) section of the paper addresses this question.

\section{A lower bound and questions.}

In this section, we show that $\aht^2$ implies $\aca$ over $\rca$. Thus, the leftmost five nodes of the last shown diagram collapse, all of them being equivalent, and really the only point of separation happens in the arrow connecting $\zrt^3$ to $\rt^2$. The following result, stated in both the language of Computability Theory and of Reverse Mathematics, is proved with much the same ideas as the classical proof that $\rt^3$ implies $\aca$\footnote{There are other examples of statements where dimension 2 suffices to prove $\aca$, such as the Canonical Ramsey's Theorem, the Regressive Function's Theorem~\cite{mileti} or the Regressive Hindman's Theorem~\cite{carlucci-mainardi}.}.

\begin{theorem}\label{lower-bound}\hfill
\begin{enumerate}
\item There exists a computable instance of $\aht^2$ whose solutions compute $0'$.
\item Over $\rca$, $\aht^2$ (equivalently, $\zrt^3$) implies $\aca$.
\end{enumerate}
\end{theorem}

\begin{proof}
We write explicitly the proof of (1), and only mention in passing that it readily formalizes to yield (2). Consider the colouring $c:\mathbb N^2\longrightarrow 2\times 2$ given by letting $c(x,y)=(i,j)$, where $i$ equals $1$ if and only if $\lambda(x)<\lambda(y)$, and $j$ equals $1$ if and only if $(0'\cap\lambda(x))[\mu(x)]=(0'\cap\lambda(x))[\mu(y)]$. Suppose that $\vec{x}=(x_n\big|n\in\mathbb N)$ is a sequence such that $\afs^2(\vec{x})$ is $c$-monochromatic. Let $n_0$ be such that $\lambda(x_{n_0})$ is as small as possible. Comparing $\lambda(x_{n_0})$ with $\lambda(x_{n_0+1})$ and $\lambda(x_{n_0+2})$, we see that either $\lambda(x_{n_0})<\lambda(x_{n_0+1})$, or $\lambda(x_{n_0+1})<\lambda(x_{n_0+2})$, or $\lambda(x_{n_0})=\lambda(x_{n_0+1})=\lambda(x_{n_0+2})$ which implies $\lambda(x_{n_0})<\lambda(x_{n_0+1}+x_{n_0+2})$. In either case, we have obtained an element of $\afs^2(\vec{x})$ whose colour is of the form $(1,j)$; hence all elements of $\afs^2(\vec{x})$ have the same colour and so $\lambda$ is strictly increasing on the sequence $\vec{x}$. In particular, this implies that $\lambda(x_n+x_{n+1}+\cdots+x_{n+l})=\lambda(x_n)$ for all $n,l$. Therefore, for any $n$ it is possible to take an $l$ large enough that $(0'\cap\lambda(x_n))[\mu(x)]=0'\cap\lambda(x_n)$ for $x=x_n+x_{n+1}+\cdots+x_{n+l}$; then any element of $\afs^2(\vec{x})$ having $x$ as its first coordinate will have colour $(1,1)$. This implies that for every $n\in\mathbb N$ we have $(0'\cap\lambda(x_n))[\mu(x_n)]=0'\cap\lambda(x_n)$ ---otherwise, taking any $t$ large enough that $(0'\cap\lambda(x_n))[\mu(y)]=0'\cap\lambda(x_n)$ for $y=x_{n+1}+\ldots+x_{n+t}$, we would obtain the element $(x_n,y)\in\afs^2(\vec{x})$ with colour $(1,0)$---. In other words, the sequence $\vec{x_n}$ is able to compute $0'$ (given any $m\in\mathbb N$, take an $n$ large enough that $\lambda(x_n)>m$ and check whether $m$ belongs to $(0'\cap\lambda(x_n))[\mu(x_n)]$).
\end{proof}

\begin{figure}[t]
\centering
\begin{equation*}
\centerline{\xymatrix{
\aca \ar@{=>}[r] \ar@{<->}[d] & \rt^2 \ar@{->}[r] & \sepzrt^2 \ar@{->}[r] \ar@{<->}[d] & \zrt^2 \ar@{->}[r] \ar@{<->}[d] & \rt^1\ar@{<->}[d] \\
\zrt^3 &  & \apaht \ar@{->}[r]  & \aht  & \text{pigeonhole} \\
}}
\end{equation*}
\caption{Diagram of implications over $\rca$. The double arrow indicates that we know the implication is not reversible.}
\label{fig:diagrama}
\end{figure}

We summarize the results of the paper, from the perspective of Reverse Mathematics, by presenting the diagram in Figure~\ref{fig:diagrama}. Whether any of the remaining arrows (except for the one connecting $\aca$ to $\rt^2$) is reversible remains an open question. For example, it is conceivable that $\aht$ already implies $\apaht$, although some ideas that have been used for similar results (e.g.~\cite[Lemma 1]{carlucci-kolodziejczyk-et-al}) do not seem to work in this context because the requirement that the sums be adjacent, thus preventing us to skip summands, poses a strong restriction. Similar difficulties arise when one attempts to obtain better lower bounds for $\zrt^2$: Carlucci~\cite[Prop. 3]{carlucci-adjacent} has found that the {\it Increasing Polarized Ramsey's Theorem for pairs}\footnote{That is, the statement that for every finite colouring $c:[\mathbb N]^2\longrightarrow k$ there are two infinite sets $X,Y$ such that all pairs of the form $\{x,y\}$ with $x\in X$, $y\in Y$, and $x<y$, have the same colour.}, denoted $\mathsf{IPT}^2$, is a lower bound for $\apaht$; the apartness condition plays a crucial r\^ole in Carlucci's proof, and we have been unable to successfully find a suitable lower bound for $\aht$ without this condition\footnote{It is conceivable that some of the ideas from~\cite{reverse-hindman-exactly-two} could be useful to obtain such lower bound results without separation/apartness. We are grateful to one of the anonymous referees for this suggestion.}. We do have a conjecture but state it in the form of a question below. A final consideration arises from thinking about versions of the principles we have stated for a bounded number of colours. For example, one could consider, e.g., the principle $\aht_k$ which states the same as $\aht$ but only for colourings with at most $k$ colours. Note that, from this viewpoint, Theorem~\ref{lower-bound} really proves that $\zrt^3_4$ implies $\aca$, but it is not clear whether that $4$ could be lowered even further. We finalize the paper by stating explicitly these questions.

\begin{question}\label{questions}\hfill
\begin{enumerate}
\item Does $\zrt^2$ (equivalently, $\aht$) imply $\sepzrt^2$ (equivalently, $\apaht$)?
\item Can one prove, under $\rca$ (or possibly under $\rca+\mathsf{B}\Sigma_2^0$) that $\aht$ (equivalently, $\zrt^2$) implies $\mathsf{D}^2$?
\item What are the provable implication relations between the different ``bounded colour'' versions of $\aht_k^n$, for various $k$ (equivalently for $\zrt_k^n$, for various $k$)?
\end{enumerate}
\end{question}

The principle $\mathsf{D}^2$ mentioned in point (2) of Questions~\ref{questions} is the statement, as considered in~\cite{dzhafarov-hirst-polarized}, that for every stable $c:[\mathbb N]^2\longrightarrow k$ (where ``stable'' means for each $n$, the function $m\longmapsto c(\{n,m\})$ is eventually constant) there exists an infinite set $X\subseteq\mathbb N$ and a colour $i<k$ such that for each $n\in X$, we have $\displaystyle{\lim_{m\to\infty}c(\{n,m\})=i}$; this statement is implied by $\mathsf{IPT}^2$ by a result of Dzhafarov and Hirst~\cite{dzhafarov-hirst-polarized}. On the other hand, regarding point (3) of Questions~\ref{questions}, it is worth noting that D. Tavernelli~\cite[Section 2.3]{tavernelli} has obtained results showing that the versions of $\aht$ with more colours are not Weihrauch-reducible to the versions with fewer colours, although this still leaves open the possibility of having some version of $\aht$ with more colours being {\it provable} (over $\rca$, or maybe over $\rca+\mathsf{B}\Sigma_2^0$) from another version with fewer colours.


\begin{thebibliography}{brot-cao-fernandez}

\bibitem{blass-hirst-simpson}
A. R. Blass, J. L. Hirst, and S. G. Simpson,
{\em Logical analysis of some theorems of combinatorics and topological dynamics.}
Contemp. Math.
{\bf 65} (1987), 125--156.

\bibitem{carlucci-adjacent}
L. Carlucci,
{\em A weak variant of Hindman's Theorem stronger than Hilbert's theorem.}
Arch. Math. Logic
{\bf 57} (2018), 381--389.

\bibitem{carlucci-overview}
L. Carlucci,
{\em Restrictions of Hindman's Theorem: An Overview.}
In L. De Mol. et al. (Eds.),
``Connecting with Computability'', Proceedings of CiE2021,
LNCS 12813, pp. 94--105.

\bibitem{carlucci-mainardi}
L. Carlucci and L. Mainardi,
{\em Regressive versions of Hindman's theorem.}
Arch. Math. Logic
{\bf 63} (2024), 447--472.

\bibitem{reverse-hindman-exactly-two}
B. F. Csima, D. D. Dzhafarov, D. R. Hirschfeldt, K. G. Jockusch, Jr., R. Solomon, and L. Brown Westrick,
{\em The reverse mathematics of Hindman's Theorem for sums of exactly two elements.}
Computability
{\bf 8} (2018), 253--263.

\bibitem{dzhafarov-hirst-polarized}
D. D. Dzhafarov and J. L. Hirst,
{\em The polarized Ramsey's theorem.}
Arch. Math. Logic
{\bf 48} (2009), 141--157.


\bibitem{dzhafarov-mummert-textbook}
D. D. Dzhafarov and C. Mummert,
{\em Reverse Mathematics. Problems, Reductions, and Proofs.}
Theory and Applications of Computability,
Springer, 2022.


\bibitem{carlucci-kolodziejczyk-et-al}
L.Carlucci, L. A. Ko\l{}odziejczyk, F. Lepore, and K. Zdanowski,
{\em New Bounds on the Strength of Some Restrictions of Hindman’s Theorem.}
In J. Kari, F. Manea, and I. Petre (eds.), ``Unveiling Dynamics and Complexity. 13th Conference on Computability in Europe, CiE 2017.''
Springer, 2017 (pp. 210--220).

\bibitem{hirschfeldt}
D. R. Hirschfeldt,
{\em Slicing the Truth. On the Computable and Reverse Mathematics of Combinatorial Principles.}
Lecture Note Series 28.
Institute for Mathematical Sciences, National University of Singapore, 2015.

\bibitem{jockusch}
C. G. Jockusch, Jr.,
{\em Ramsey’s Theorem and recursion theory.}
J. Symbolic Logic
{\bf 37} (1972), 268--280.

\bibitem{mileti}
J. R. Mileti,
{\em The Canonical Ramsey Theorem and Computability Theory.}
Trans. Amer. Math. Soc.
{\bf 360} (2008), 1309--1340.

\bibitem{milliken-thm}
K. R. Milliken,
{\em Ramsey's Theorem with Sums or Unions.}
J. Combin. Theory Ser. A
{\bf 18} (1975), 276--290.

\bibitem{petrenko-protasov}
O. Petrenko and I. Protasov,
{\it Selective and Ramsey Ultrafilters on $G$-spaces}.
Notre Dame J. Form. Log.
{\bf 58}~no. 3 (2017), 453--459.

\bibitem{simpson-subsystems}
S. G. Simpson,
{\em Subsystems of Second Order Arithmetic. Second Edition.}
Perspectives in Logic, ASL,
Cambridge University Press, 2009.

\bibitem{tavernelli}
D. Tavernelli,
{\it On the strength of restrictions of Hindman's theorem}.
Master's Degree (Corso di Laurea Magistrale) Thesis,
Sapienza Universit\'a di Roma, 2018.

\bibitem{taylor-thm}
A. D. Taylor,
{\em A Canonical Partition Relation for Finite Subsets of $\omega$.}
J. Combin. Theory Ser. A
{\bf 21} (1976), 137--146.

\end{thebibliography}
\end{document}